\documentclass{amsart}

\usepackage[latin1]{inputenc}
\usepackage{amssymb,amsmath,amsthm}
\usepackage{amsfonts}
\usepackage{ifthen}
\usepackage[neverdecrease,pointedenum]{paralist}
\usepackage{color}
\usepackage{multirow}
\usepackage{bigdelim}
\usepackage{rotating}
\usepackage{longtable}

\newcommand{\vect}[1]{\ensuremath{\mathbf{#1}}}
\newcommand{\card}[1]{\ensuremath{\lvert{#1}\rvert}}
\newcommand{\cl}[1]{\ensuremath{\mathcal{#1}}}
\newcommand{\nset}[1]{\ensuremath{[{#1}]}}
\newcommand{\couples}[1][\nset{n}]{\ensuremath{\binom{#1}{2}}} % the set of 2-element subsets of an n-element set
\newcommand{\multiset}[1]{\ensuremath{\langle{#1}\rangle}} % multiset

\DeclareMathOperator{\deck}{deck}                % deck
\DeclareMathOperator{\setdeck}{set-deck}         % set-deck
\DeclareMathOperator{\pr}{pr}                    % projection
\DeclareMathOperator{\id}{id}                    % identity map
\DeclareMathOperator{\supp}{supp}                % supp
\DeclareMathOperator{\oddsupp}{oddsupp}          % oddsupp
\DeclareMathOperator{\Clo}{Clo}                  % clone

\theoremstyle{plain}
\newtheorem{theorem}{Theorem}[section]
\newtheorem{proposition}[theorem]{Proposition}
\newtheorem{lemma}[theorem]{Lemma}
\newtheorem{corollary}[theorem]{Corollary}

\theoremstyle{definition}

\theoremstyle{remark}
\newtheorem{remark}[theorem]{Remark}

\usepackage{tikz}
\newcommand{\PostsLattice}[1]{
  \begin{tikzpicture}[scale=#1, transform shape]
    \tikzstyle{every node} = [circle, fill=black,scale=0.5]
    \tikzstyle{every label} = [scale=2,draw=none, fill=none, label distance=-4]
    \node (P2) [label=above:$\Omega$] at (7.0+1.0   ,10.5+2.5+1.2) {};
    \node (T0) [label=above left:$T_0$]  at (7.0-3.0,10.5-.5+2.5+.85) {};
    \node (T1) [label=above right:$T_1$] at (7.0+3.0,10.5-.5+2.5+.85) {};
    \node (T)  %[label=above:$T_c$]
          at (7.0-1.0,10.5-1.0+3) {};
    \node (M)  [label=above left:$M$] at (6.76+.1+.4   ,9.38+1.1) {};
    \node (T0M)%[label=left:$M_0\;$]
          at (7.23-1.5-.4+.2-.2, 9.51-.5+1-.1) {};
    \node (T1M)%[label=right:$\;M_1$]
          at (6.76+1.0+.4+.3+.2, 9.38-.5+1-.1) {};
    \node (TM) %[label=above right:\raisebox{1.25ex}{$M_c$}]
          at (7.23-.5, 9.51-1.0+1-.3) {};

    \node (L)  [label=above right:$L$] at (7.0+.5   , 5.0-.2+.2) {};
    \node (T0L)%[label=left:$L_0$]
          at (7.0-1.6, 5.0-.2-.5) {};
    \node (T1L)%[label=right:$L_1$]
          at (7.0+1.6, 5.0-.2-.5) {};
    \node (TL)% [label=below left:$L_c$]
          at (7.0-.5, 5.0-.2-1.0-.2) {};
    \node (SL) %[label=right:$LS$]
          at (7.0+0.0, 5.0-.2-0.5) {};
    
    \node (S) [label=above right:$S$] at (7+.1,7.0-0.05+.4) {};
    \node (ST) %[label=left:$S_c$]
          at (6.5,6.5-0.2+.2) {};
    \node (SM) %[label=below left:$SM$]
          at (6-.1,6.0-0.35) {};

    \node (P21) [label=right:$\Omega(1)$] at (7.0+0.5,3.5-1) {};
    \node (SP21)%[label=right:$I^*$]
          at (7.0-0.2, 3.5-2) {};
    \node (AV)  %[label=below:$I$]
          at (7.0 +.5  , 3.0-4 +.75) {};
    \node (T0AV)%[label=below left:$I_0$]
          at (7.0-2, 3.0-4) {};
    \node (T1AV)%[label=below right:$I_1$]
          at (7.0+2, 3.0-4) {};
    \node (TAV) %[label=below:$I_c$]
          at (7.0-.5,  3.0-4-.75) {}; % I

    \foreach \from/\to in {
          P2/T0, P2/T1, T0/T, T1/T,
          M/T0M, M/T1M, T0M/TM, T1M/TM,
          P2/M, T0/T0M, T1/T1M, T/TM,
          L/T0L, L/T1L, T0L/TL, T1L/TL,
          P2/L, T0/T0L, T1/T1L, T/ST, ST/TL,
          AV/T0AV, AV/T1AV, T0AV/TAV, T1AV/TAV,
          L/P21, P21/AV, T0L/T0AV, T1L/T1AV, TL/TAV,
          P2/S, S/ST, S/SL, L/SL, SL/TL, SL/SP21, P21/SP21, SP21/TAV,
          ST/SM, SM/TAV}
    \draw [-] (\from) -- (\to);

    \node (T02)  [label=left:$U_2$] at (0, 8.5) {};
    \node (T02T) %[label=left:$T_cU_2$]
          at (0+1.0, 8.5-.7) {};
    \node (T02M) %[label=right:$MU_2$]
          at (0+2.0, 8.5-.3) {};
    \node (T02TM)%[label=right:$M_cU_2$]
          at (0+3.0, 8.5-1.0) {};
    \node (T03)  [label=left:$U_3$] at (0, 7.0) {};
    \node (T03T) %[label=left:$T_cU_3$]
          at (0+1.0, 7.0-.7) {};
    \node (T03M) %[label=right:$MU_3$]
          at (0+2.0, 7.0-.3) {};
    \node (T03TM)%[label=right:$M_cU_3$]
          at (0+3.0, 7.0-1.0) {};
    \node (T0H) [draw=none, fill=none, scale=0.1] at (0, 6.5) {};
    \node (T0HT) [draw=none, fill=none, scale=0.1] at (0+1.0, 6.5-.7) {};
    \node (T0HM) [draw=none, fill=none, scale=0.1] at (0+2.0, 6.5-.3) {};
    \node (T0HTM) [draw=none, fill=none, scale=0.1] at (0+3.0, 6.5-1.0) {};
    \node (T0e)  [label=left:$U_\infty$] at (0, 5.0) {};
    \node (T0eT) %[label=left:$T_cU_\infty$]
          at (0+1.0, 5.0-.7) {};
    \node (T0eM) %[label=right:$MU_\infty$]
          at (0+2.0, 5.0-.3) {};
    \node (T0eTM)%[label=right:$M_cU_\infty$]
          at (0+3.0, 5.0-1.0) {};
    
    \node (A) [label=above right:$\Lambda$] at (0+3.5,3.5-.5) {};
    \node (AT1) %[label=right:$\Lambda_1$]
          at (0+4.5,3.0-.5) {};
    \node (AT0) %[label=left:$\Lambda_0$]
          at (0+2.0,3.0-.5) {};
    \node (AT) %[label=below:$\Lambda_c$]
          at (0+3.0,2.5-.5) {}; 

    \node (T12)  [label=right:$W_2$] at (14, 8.5) {};
    \node (T12T) %[label=right:\raisebox{-3ex}{$\!T_cW_2$}]
          at (14-1.0, 8.5-.7) {};
    \node (T12M) %[label=left:$MW_2\;$]
          at (14-2.0, 8.5-.3) {};
    \node (T12TM)%[label=left:$M_cW_2$]
          at (14-3.0, 8.5-1.0) {};
    \node (T13)  [label=right:$W_3$] at (14, 7.0) {};
    \node (T13T) %[label=right:\raisebox{-3ex}{$\!T_cW_3$}]
          at (14-1.0, 7.0-.7) {};
    \node (T13M) %[label=left:$MW_3\;$]
          at (14-2.0, 7.0-.3) {};
    \node (T13TM)%[label=left:$M_cW_3$]
          at (14-3.0, 7.0-1.0) {};
    \node (T1H) [draw=none, fill=none, scale=0.1] at (14, 6.5) {};
    \node (T1HT) [draw=none, fill=none, scale=0.1] at (14-1.0, 6.5-.7) {};
    \node (T1HM) [draw=none, fill=none, scale=0.1] at (14-2.0, 6.5-.3) {};
    \node (T1HTM) [draw=none, fill=none, scale=0.1] at (14-3.0, 6.5-1.0) {};
    \node (T1e)  [label=right:$W_\infty$] at (14, 5.0) {};
    \node (T1eT) %[label=right:\raisebox{-3ex}{$\!T_cW_\infty$}]
          at (14-1.0, 5.0-.7) {};
    \node (T1eM) %[label=left:$MW_\infty\;$]
          at (14-2.0, 5.0-.3) {};
    \node (T1eTM)%[label=left:$M_cW_\infty$]
          at (14-3.0, 5.0-1.0) {};

    \node (V) [label=above left:$V$] at (14-3.5,3.5-.5) {};
    \node (VT0) %[label=left:$V_0$]
          at (14-4.5,3.0-.5) {};
    \node (VT1) %[label=right:$V_1$]
          at (14-2.0,3.0-.5) {};
    \node (VT) %[label=below:$V_c$]
          at (14-3.0,2.5-.5) {}; 

    \foreach \from/\to in {
          T02/T02T, T02/T02M, T02T/T02TM, T02M/T02TM,
          T03/T03T, T03/T03M, T03T/T03TM, T03M/T03TM,
          T0e/T0eT, T0e/T0eM, T0eT/T0eTM, T0eM/T0eTM,
          T12/T12T, T12/T12M, T12T/T12TM, T12M/T12TM,
          T13/T13T, T13/T13M, T13T/T13TM, T13M/T13TM,
          T1e/T1eT, T1e/T1eM, T1eT/T1eTM, T1eM/T1eTM,
          T02/T03, T02T/T03T, T02M/T03M, T02TM/T03TM,
          T03/T0H, T03T/T0HT, T03M/T0HM, T03TM/T0HTM,
          T12/T13, T12T/T13T, T12M/T13M, T12TM/T13TM,
          T13/T1H, T13T/T1HT, T13M/T1HM, T13TM/T1HTM,
          T0eM/AT0, T0eTM/AT,
          T1eM/VT1, T1eTM/VT,
          A/AT0, A/AT1, AT0/AT, AT1/AT,
          V/VT0, V/VT1, VT0/VT, VT1/VT,
          A/AV, AT0/T0AV, AT1/T1AV, AT/TAV,
          V/AV, VT0/T0AV, VT1/T1AV, VT/TAV,
          T02TM/SM, T12TM/SM,
          T0/T02, T0M/T02M, T/T02T, TM/T02TM,
          T1/T12, T1M/T12M, T/T12T, TM/T12TM}
    \draw [-] (\from) -- (\to);
    \path (M) edge [out=215, in=90] (A);
    \path (T1M) edge [out=215, in=90] (AT1);
    \path (M) edge [out=325, in=90] (V);
    \path (T0M) edge [out=325, in=90] (VT0);

    \foreach \from/\to in {
          T0H/T0e, T0HT/T0eT, T0HM/T0eM, T0HTM/T0eTM,
          T1H/T1e, T1HT/T1eT, T1HM/T1eM, T1HTM/T1eTM}
    \draw [dotted] (\from) -- (\to);
  \end{tikzpicture}
}

\begin{document}
\title{Set-reconstructibility of Post classes}

\author{Miguel Couceiro}
\address[M. Couceiro]{LAMSADE -- CNRS \\
Universit\'e Paris-Dauphine \\
Place du Mar\'echal de Lattre de Tassigny \\
75775 Paris Cedex 16 \\
France}
\email{miguel.couceiro@dauphine.fr}

\author{Erkko Lehtonen}
\address[E. Lehtonen]{University of Luxembourg \\
Computer Science and Communications Research Unit \\
6, rue Richard Coudenhove-Kalergi \\
L--1359 Luxembourg \\
Luxembourg}
\email{erkko.lehtonen@uni.lu}

\author{Karsten Sch\"olzel}
\address[K. Sch\"olzel]{University of Luxembourg \\
Mathematics Research Unit \\
6, rue Richard Coudenhove-Kalergi \\
L--1359 Luxembourg \\
Luxembourg}
\email{karsten.schoelzel@uni.lu}

\date{\today}

\begin{abstract}
The clones of Boolean functions are classified in regard to set-reconstructibility via a strong dichotomy result:
the clones containing only affine functions, conjunctions, disjunctions or constant functions are set-re\-con\-struct\-ible,
whereas the remaing clones are not weakly reconstructible.
\end{abstract}

\maketitle

%%%%%%%%%%%%%%%%%%%%%%%%%%%%%%%%%%%%%%%%%%%%%%%%%%

\section{Introduction}
\label{sec:intro}

Reconstruction problems have been considered in various fields of mathematics and theoretical computer science,
and they share the same meta-formulation: \emph{given a family of ``objects'' and a systematic way of forming some sort of ``derived objects'',
is an object uniquely determined (up to a sort of equivalence) by the collection of its derived objects?}
It is possible that the same derived object arises from a given object in many different ways, and we usually keep track of the number of times each derived object arises; in other words, ``collection'' means the \textbf{multiset} of derived objects.
On the other hand, if we ignore the numbers of occurrences of the derived objects, i.e., we take ``collection'' to mean the \textbf{set} of derived objects, then we are dealing with what is referred to as a set-reconstruction problem.

Several instances of this general formulation have become celebrated conjectures that have attracted a great deal of attention within 
the scientific community. Among these, the graph reconstruction conjecture (in both variants of vertex- or edge-deletion) 
\cite{Kelly1942,Ulam} remains one of the most 
challenging that has survived as an open problem for many decades.  Nonetheless, it has been shown to hold for numerous classes
of graphs such as trees, regular graphs, etc. In fact, Bollob\'as \cite{Bollo} 
showed that the probability of finding a non-reconstructible graph tends to 0 as the number of vertices tends to infinity.
In some other noteworthy instances, e.g., for directed graphs and hypergraphs,
reconstructibility has been shown not to hold in general;
see, e.g., 
\cite{Kocay,KocLui,Stockmeyer}. 

In this paper we consider a reconstruction problem for functions of several arguments,
taking the identification of a pair of arguments as the way of forming derived objects:
is a function $f \colon A^n \to B$ determined (up to equivalence) by
its identification minors?

Lehtonen~\cite{LehtonenSymmetric,LehtonenLinear} answers this question positively for certain function classes such as those of symmetric functions or
affine functions. Recently, we showed \cite{CLSmonotone} that the class of order-preserving functions is not reconstructible, 
even if restricted to lattice polynomial functions. 
In the case of Boolean functions, the latter result refines into a classification of Post classes
(clones of Boolean functions):
the only reconstructible Post classes are the ones containing only affine functions, conjunctions, disjunctions or constant functions.
The remaining Post classes are not weakly reconstructible.
 
The purpose of this paper is to make this dichotomy of Post classes even more contrasting: the reconstructible Post classes are actually set-reconstructible. This shows that reconstructibility is the same as set-reconstructibility in this setting.
 
The paper is organized as follows. In Section \ref{sec:preliminaries} we recall the basic notions, 
state preliminary results and formulate the reconstruction problem
for functions of several arguments and identification minors.
We focus on the (set)-reconstructibility of clones of operations in Section \ref{sec:clones}, where we 
provide a dichotomy theorem dealing with the set-reconstructibility of Post classes.

%%%%%%%%%%%%%%%%%%%%%%%%%%%%%%%%%%%%%%%%%%%%%%%%%%

\section{Preliminaries}
\label{sec:preliminaries}

\subsection{General}
\label{sec:preliminaries:general}

Let $\mathbb{N} := \{0, 1, 2, \dots\}$.
Throughout this paper, $k$, $\ell$, $m$ and $n$ stand for positive integers, and $A$ and $B$ stand for arbitrary finite sets with at least two elements. The set $\{1, \dots, n\}$ is denoted by $\nset{n}$. The set of all $2$-element subsets of a set $A$ is denoted by $\couples[A]$. Tuples are denoted by bold-face letters and components of a tuple are denoted by the corresponding italic letters with subscripts, e.g., $\vect{a} = (a_1, \dots, a_n)$.

Let $\vect{a} \in A^n$, and let $\sigma \colon \nset{m} \to \nset{n}$. We will write $\vect{a} \sigma$ to denote the $m$-tuple $(a_{\sigma(1)}, \dots, a_{\sigma(m)})$.
Since the $n$-tuple $\vect{a}$ can be formally seen as the map $\vect{a} \colon \nset{n} \to A$, $i \mapsto a_i$, the $m$-tuple $\vect{a} \sigma$ is just the composite map $\vect{a} \circ \sigma \colon \nset{m} \to A$.

A \emph{finite multiset} $M$ on a set $S$ is a couple $(S, \mathbf{1}_M)$, where $\mathbf{1}_M \colon S \to \mathbb{N}$ is a map, called a \emph{multiplicity function,} such that the set $\{x \in S : \mathbf{1}_M(x) \neq 0\}$ is finite. Then the sum $\sum_{x \in S} \mathbf{1}_M(x)$ is a well-defined natural number, and it is called the \emph{cardinality} of $M$. For each $x \in S$, the number $\mathbf{1}_M(x)$ is called the \emph{multiplicity} of $x$ in $M$.
If $(a_i)_{i \in I}$ is a finite indexed family of elements of $S$, then we will write $\multiset{a_i : i \in I}$ to denote the multiset in which the multiplicity of each $x \in S$ equals $\card{\{i \in I : a_i = x\}}$.

\subsection{Functions of several arguments and identification minors}

A \emph{function} (\emph{of several arguments}) from $A$ to $B$ is a map $f \colon A^n \to B$ for some positive integer $n$, called the \emph{arity} of $f$. Functions of several arguments from $A$ to $A$ are called \emph{operations} on $A$. Operations on $\{0, 1\}$ are called \emph{Boolean functions.}
We denote the set of all $n$-ary functions from $A$ to $B$ by $\cl{F}_{AB}^{(n)}$, and we denote the set of all functions from $A$ to $B$ of any finite arity by $\cl{F}_{AB}$. We also write $\cl{F}_{AB}^{(\geq n)}$ for $\bigcup_{m \geq n} \cl{F}_{AB}^{(m)}$. In other words, $\cl{F}_{AB}^{(n)} = B^{A^n}$ and $\cl{F}_{AB} = \cl{F}_{AB}^{(\geq 1)}$.
We also denote by $\cl{O}_A$ the set of all operations on $A$.
For any class $\mathcal{C} \subseteq \cl{F}_{AB}$, we let $\mathcal{C}^{(n)} := \mathcal{C} \cap \cl{F}_{AB}^{(n)}$ and $\mathcal{C}^{(\geq n)} := \mathcal{C} \cap \cl{F}_{AB}^{(\geq n)}$.

Let $f \colon A^n \to B$. For $i \in \nset{n}$, the $i$-th argument of $f$ is \emph{essential,} or $f$ \emph{depends} on the $i$-th argument, if there exist tuples $\vect{a}, \vect{b} \in A^n$ such that $a_j = b_j$ for all $j \in \nset{n} \setminus \{i\}$ and $f(\vect{a}) \neq f(\vect{b})$. Arguments that are not essential are \emph{inessential.}

We say that a function $f \colon A^n \to B$ is a \emph{minor} of another function $g \colon A^m \to B$, and we write $f \leq g$, if there exists a map $\sigma \colon \nset{m} \to \nset{n}$ such that $f(\vect{a}) = g(\vect{a} \sigma)$ for all $\vect{a} \in A^m$.
The minor relation $\leq$ is a quasiorder on $\cl{F}_{AB}$, and, as for all quasiorders, it induces an equivalence relation on $\cl{F}_{AB}$ by the following rule: $f \equiv g$ if and only if $f \leq g$ and $g \leq f$. We say that $f$ and $g$ are \emph{equivalent} if $f \equiv g$. Furthermore, $\leq$ induces a partial order on the quotient $\cl{F}_{AB} / {\equiv}$. (Informally speaking, $f$ is a minor of $g$, if $f$ can be obtained from $g$ by permutation of arguments, addition of inessential arguments, deletion of inessential arguments, and identification of arguments. If $f$ and $g$ are equivalent, then each one can be obtained from the other by permutation of arguments, addition of inessential arguments, and deletion of inessential arguments.)
We denote the $\equiv$-class of $f$ by $f / {\equiv}$.
Note that equivalent functions have the same number of essential arguments and every nonconstant function is equivalent to a function with no inessential arguments.
Note also in particular that if $f, g \colon A^n \to B$, then $f \equiv g$ if and only if there exists a bijection $\sigma \colon \nset{n} \to \nset{n}$ such that $f(\vect{a}) = g(\vect{a} \sigma)$ for all $\vect{a} \in A^n$.

Of particular interest to us are the following minors. Let $n \geq 2$, and let $f \colon A^n \to B$. For each $I \in \couples$, we define the function $f_I \colon A^{n-1} \to B$ by the rule
$f_I(\vect{a}) = f(\vect{a} \delta_I)$ for all $\vect{a} \in A^{n-1}$,
where $\delta_I \colon \nset{n} \to \nset{n - 1}$ is given by the rule
\[
\delta_I(i) =
\begin{cases}
i, & \text{if $i < \max I$,} \\
\min I, & \text{if $i = \max I$,} \\
i - 1, & \text{if $i > \max I$.}
\end{cases}
\]
In other words, if $I = \{i, j\}$ with $i < j$, then
\[
f_I(a_1, \dots, a_{n-1}) =
f(a_1, \dots, a_{j-1}, a_i, a_j, \dots, a_{n-1}).
\]
Note that $a_i$ occurs twice on the right side of the above equality: both at the $i$-th and at the $j$-th position. We will refer to the function $f_I$ as an \emph{identification minor} of $f$. This nomenclature is motivated by the fact that $f_I$ is obtained from $f$ by identifying the arguments indexed by the pair $I$.

\subsection{Reconstruction problem for functions and identification minors}

Let $f \colon A^n \to B$.
We will refer to the equivalence classes $f_I / {\equiv}$ of the identification minors $f_I$ of $f$ ($I \in \couples$) as the \emph{cards} of $f$.
The \emph{deck} of $f$, denoted $\deck f$, is the multiset $\multiset{f_I / {\equiv} : I \in \couples}$ of the cards of $f$.

We may now ask whether a function $f$ is uniquely determined, up to equivalence, by its deck.
In order to discuss whether and to which extent this is the case, we will use the following terminology that is more or less standard.

We say that a function $f' \colon A^n \to B$ is a \emph{reconstruction} of $f$, or that $f$ and $f'$ are \emph{hypomorphic,} if $\deck f = \deck f'$, or, equivalently, if there exists a bijection $\phi \colon \couples \to \couples$ such that $f_I \equiv f'_{\phi(I)}$ for every $I \in \couples$.
If the last condition holds with $\phi$ equal to the identity map on $\couples$, i.e., if $f_I \equiv f'_I$ for every $I \in \couples$, then we say that $f$ and $f'$ are \emph{strongly hypomorphic.}
Note that strongly hypomorphic functions are necessarily hypomorphic, but the converse is not true in general.

A function is \emph{reconstructible} if it is equivalent to all of its reconstructions.
A class $\mathcal{C} \subseteq \mathcal{F}_{AB}$ of functions is \emph{reconstructible} if all members of $\mathcal{C}$ are reconstructible.
A class $\mathcal{C}$ is \emph{weakly reconstructible} if for every $f \in \mathcal{C}$, all reconstructions of $f$ that are members of $\mathcal{C}$ are equivalent to $f$.
A class $\mathcal{C}$ is \emph{recognizable} if all reconstructions of the members of $\mathcal{C}$ are members of $\mathcal{C}$.
Note that a reconstructible class is necessarily weakly reconstructible, but the converse is not true in general.
If a class is recognizable and weakly reconstructible, then it is reconstructible.

We summarize here some of the known results on the reconstruction problem for functions and identification minors.

\begin{theorem}[{\cite[Thm.~5.1]{LehtonenSymmetric}}]
\label{thm:totsymm}
Assume that $n \geq k + 2$ and $\card{A} = k$. If $f \colon A^n \to B$ is totally symmetric, then $f$ is reconstructible.
\end{theorem}

\begin{theorem}[{\cite[Prop.~5.2]{LehtonenSymmetric}}]
Assume that $n > \max(k, 3)$ and $\card{A} = k$.
The class of totally symmetric functions $f \colon A^n \to B$ is weakly reconstructible.
\end{theorem}

\begin{theorem}[{\cite[Thm.~4.7]{LehtonenLinear}}]
\label{thm:affine}
Let $(G; +, \cdot)$ be a finite field of order $q$. The affine functions of arity at least $\max(q, 3) + 1$ over $(G; +, \cdot)$ are reconstructible.
\end{theorem}

Note that the deck of a function is defined as the \textbf{multiset} of its cards. Considering instead the \textbf{set} of cards, we obtain a variant of the reconstruction problem, the so-called set-reconstruction problem.
The \emph{set-deck} of a function $f \colon A^n \to B$, denoted $\setdeck f$, is the set $\{f_I / {\equiv} : I \in \couples\}$ of the cards of $f$.
A function $f' \colon A^n \to B$ is a \emph{set-reconstruction} of $f$ if $\setdeck f = \setdeck f'$.
A function is \emph{set-reconstructible} if it is equivalent to all of its set-reconstructions.
In an analogous way, we can define \emph{set-reconstructibility,} \emph{weak set-reconstructibility} and \emph{set-recognizability} for classes of functions.

\begin{remark}
Set-reconstructibility implies reconstructibility.
The converse does not hold in general.
\end{remark}

We say that a function $f \colon A^n \to B$ has a \emph{unique identification minor} if $f_I \equiv f_J$ for all $I, J \in \couples$.
The following three conditions are clearly equivalent:
\begin{enumerate}[\rm (i)]
\item $f$ is a function with a unique identification minor, and $g$ is the unique identification minor of $f$.
\item $\deck f = \multiset{g : I \in \couples}$.
\item $\setdeck f = \{g\}$.
\end{enumerate}

\begin{lemma}
\label{lem:recsetrec}
A function with a unique identification minor is reconstructible if and only if it is set-reconstructible.
\end{lemma}

\begin{proof}
If $f$ is a function with a unique identification minor, then $f'$ is a reconstruction of $f$ if and only if $f'$ is a set-reconstruction of $f$. The claim thus follows.
\end{proof}

\section{Reconstruction problem for clones}
\label{sec:clones}

If $f \colon B^n \to C$ and $g_1, \dots, g_n \colon A^m \to B$, then the \emph{composition} of $f$ with 
$g_1, \dots, g_n$ is the function $f(g_1, \dots, g_n) \colon A^m \to C$ given by the rule
\[
f(g_1, \dots, g_n)(\vect{a}) =
f \bigl( g_1(\vect{a}), \dots, g_n(\vect{a}) \bigr),
\]
for all $\vect{a} \in A^m$.

For integers $n$ and $i$ such that $1 \leq i \leq n$, the $i$-th $n$-ary \emph{projection} on $A$ is the operation $\pr_i^{(n)} \colon A^n \to A$, $(a_1, \dots, a_n) \mapsto a_i$ for all $(a_1, \dots, a_n) \in A^n$. 

A \emph{clone} on $A$ is a class of operations on $A$ that contains all projections on $A$ and is closed under functional composition. Trivial examples of clones are the set $\cl{O}_A$ of all operations on $A$ and the set of all projections on $A$.

The clones on the two-element set $\{0, 1\}$ were completely described by Post~\cite{Post},
and they constitute a countably infinite lattice, known as Post's lattice (see Figure~\ref{fig:PostsLattice}).

\begin{figure}
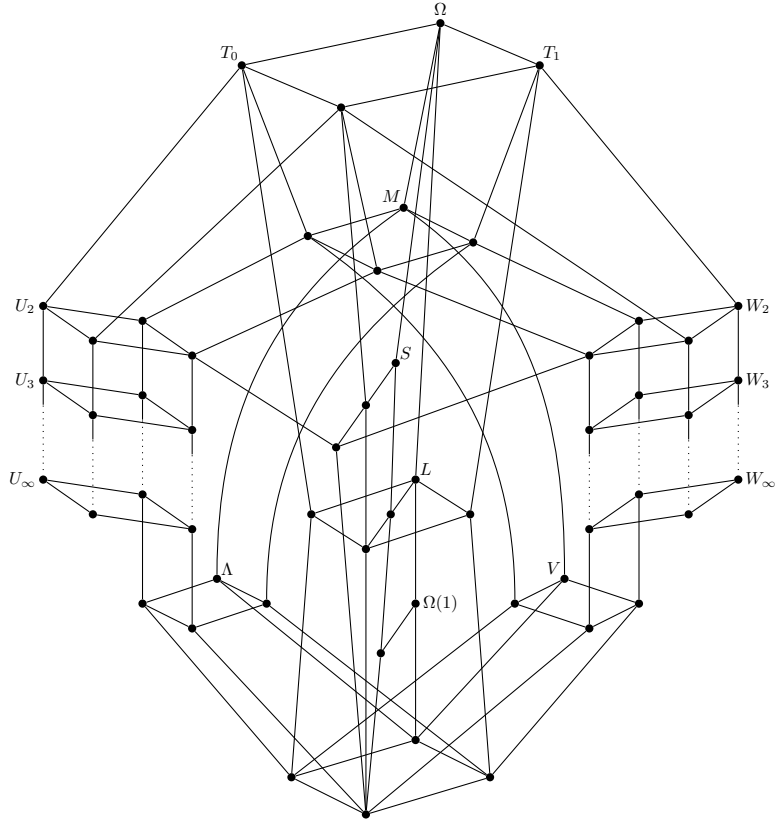

\PostsLattice{0.66}
\caption{Post's lattice.}
\label{fig:PostsLattice}
\end{figure}

In the sequel, we will make specific reference to the following clones of Boolean functions:
\begin{itemize}
\item the clone $\Lambda$ of polynomial operations of the two-element meet-semilattice,
\item the clone $V$ of polynomial operations of the two-element join-semilattice,
\item the clone $L$ of polynomial operations of the group of addition modulo $2$.
\end{itemize}

In~\cite{CLSmonotone}, we have classified the clones of Boolean functions in regard to constructibility.
It is remarkable that the clones that are not reconstructible are not even weakly reconstructible; moreover, such clones contain an infinity of functions that cannot be reconstructed even if we knew the two-element subset of $\nset{n}$ giving rise to each one of its cards.

\begin{theorem}[{\cite[Thm.~5.4]{CLSmonotone}}]
\label{thm:recPost}
Let $\mathcal{C}$ be a clone on $\{0, 1\}$.
If $\mathcal{C}$ is included in $\Lambda$, $V$ or $L$, then $\mathcal{C}^{(\geq 4)}$ is reconstructible.
Otherwise, $\mathcal{C}^{(\geq n)}$ is not weakly reconstructible for any $n \geq 1$, and $\mathcal{C}$ contains pairs of nonequivalent strongly hypomorphic functions of arbitrarily high arity.
\end{theorem}

Our aim is to make this dichotomy even more contrasting. We will show that the clones that are reconstructible are actually set-reconstructible.

Let $\circ$ be a binary operation on $A$.
For a positive integer $\ell$, define the operation $\circ_\ell$ of arity $\ell$ by the following recursion:
$\circ_1 := \id_A$,
and for $\ell \geq 2$, we let
\[
\circ_{\ell}(a_1, \dots, a_\ell) = \circ_{\ell - 1}(a_1, \dots, a_{\ell - 1}) \circ a_\ell
\]
for all $a_1, \dots, a_\ell \in A$.
For integers $n$ and $\ell$ such that $1 \leq \ell \leq n$, define the operation $\circ^{(n)}_\ell \colon A^n \to A$ by the rule $\circ^{(n)}_\ell(a_1, \dots, a_n) = \circ_\ell(a_1, \dots, a_\ell)$ for all $a_1, \dots, a_n \in A$.
It is easy to see that if $(A; \circ)$ is a semilattice (i.e., an associative, commutative and idempotent groupoid) or a Boolean group (i.e., a group in which $x + x = 0$ holds for every $x$), then the set of all nonconstant term operations of $(A; {\circ})$ equals $\{f : \text{$f \equiv \circ^{(n)}_\ell$ for some $\ell$, $n$}\}$. In the case of a Boolean group, the constant function $0$ is also a term operation of $(A; \circ)$.

\begin{remark}
The clones $\Lambda$, $V$ and $L$ are of the form $\Clo(\circ)$ for a commutative semigroup $(A; \circ)$.
Namely, $\Lambda$ and $V$ correspond to the case when $(A; \circ)$ is a semilattice, and $L$ corresponds to the case when $(A; \circ)$ is a Boolean group.
\end{remark}

\begin{lemma}
\label{lem:4-aryBoolean}
Assume that $(\{0,1\}; {\circ})$ is a semilattice or a Boolean group.
For $r \in \{2, 3\}$, no set-reconstruction of $\circ^{(4)}_r$ depends on all of its arguments.
\end{lemma}

\begin{proof}
Consider first the case that $\circ$ is a semilattice operation.
Then $\setdeck \circ^{(4)}_r = \{\circ_r, \circ_{r-1}\}$.
Observe that $\circ_r$ is order-preserving and so are all of its minors.
If $n \geq \card{A} + 2$ and $g \colon A^n \to A$ is a function that is not order-preserving, then $g$ has an identification minor that is not order-preserving (see Example~3.17 in~\cite{LehtonenSymmetric}).
Therefore, any set-reconstruction of $\circ_r$ must be order-preserving.
We can read off of the tables in Appendix~B of \cite{CLSmonotone} (and it is not difficult to verify) that the only monotone Boolean function of arity $4$ that has the same set-deck as $\circ^{(4)}_r$ is $\circ^{(4)}_r$ itself.
We conclude that $\circ^{(4)}_r$ is set-reconstructible.

Consider then the case that $\circ$ is a Boolean group operation.
Then $\setdeck \circ^{(4)}_r = \{\circ_r, \circ_{r-2}\}$, where $\circ_0$ stands for the constant function $0$.
Let $g$ be a reconstruction of $f$.
Suppose, on the contrary, that $g$ depends on all of its arguments.
Let $p$ be the unique multilinear polynomial representation (the Zhegalkin polynomial) of $g$.
Obviously $g$ is not a constant function, so $p$ has some nonconstant monomials.
If $p$ has no monomial of degree greater than $1$, then $p = x_1 + x_2 + x_3 + x_4$. Consequently $g$ is totally symmetric and hence has a unique identification minor. This is not possible.

If $p$ has a monomial of degree at least $3$, then it is easy to verify that there is $I \in \couples$ such that the Zhegalkin polynomial of $g_I$ has a monomial of degree at least $2$ and is hence not equivalent to either of $\circ_r$ and $\circ_{r-2}$.
Namely, assume that $\{s, t, u, v\} = \nset{4}$.
If $x_s x_t$ is a monomial of $p$, then $g_{\{u, v\}}$ has a monomial of degree $2$.
We may thus assume that there are no monomials of degree $2$ in $p$.
If
$x_s x_t x_u$ is the only monomial of degree $3$ in $p$, or
$x_s x_t x_u$ and $x_s x_t x_u$ are the only monomials of degree $3$ in $p$, or
$x_s x_t x_u$, $x_s x_t x_v$ and $x_s x_u x_v$ are the only monomials of degree $3$ in $p$, or
every one of $x_s x_t x_u$, $x_s x_t x_v$, $x_s x_u x_v$ and $x_t x_u x_v$ is a monomial in $p$,
then the Zhegalkin polynomial of $g_{\{s, t\}}$ has a monomial of degree $2$.
We may thus assume that $p$ does not have monomials of degree $3$.

We are left with the case that $p$ comprises the monomial $x_1 x_2 x_3 x_4$ and some monomials of degree $1$. It is clear that the Zhegalkin polynomial of any identification minor of $g$ has a monomial of degree $3$.

We have reached a contradiction, and the proof is complete.
\end{proof}

We recall a useful result due to Willard.

\begin{theorem}[{Willard~\cite[Thm.~2.5]{Willard}}]
\label{thm:Willard2.6}
Assume that $k = \card{A}$ and $n \geq \max(k, 3) + 2$, and let $f \colon A^n \to B$ be a function that depends on all of its arguments.
If every $(n - 1)$-ary minor of $f$ is totally symmetric, then $f$ is determined by either $\supp$ or $\oddsupp$ and is hence totally symmetric.
\end{theorem}

The following theorem is a special case of Theorem~7 in~\cite{CouLehGSL}.

\begin{theorem}[{\cite{CouLehGSL}}]
\label{thm:GSL}
Assume that $k = \card{A}$ and $n \geq \max(k, 3) + 1$, and let $f \colon A^n \to B$.
Then all identification minors of $f$ are essentially unary if and only if $f$ is essentially unary.
Furthermore, in this case, the identification minors of $f$ are equivalent to $f$.
\end{theorem}

\begin{corollary}
\label{cor:GSL}
Assume that $k = \card{A}$ and $n \geq \max(k, 3) + 1$, and let $f \colon A^n \to B$.
If $f$ is essentially unary, then $f$ is reconstructible.
\end{corollary}

\begin{proof}
The function $f$ has a unique identification minor, namely $f$. It follows from Theorem~\ref{thm:GSL} that any reconstruction of $f$ is equivalent to $f$, that is, $f$ is reconstructible.
\end{proof}

\begin{proposition}
\label{prop:setrecclones}
Assume that $k = \card{A}$ and $(A; {\circ})$ is a semilattice or a Boolean group.
If $\mathcal{C}$ is a subclone of $\Clo({\circ})$, then the class
$\mathcal{C}^{(\geq k + 2)}$
is set-reconstructible.
\end{proposition}

\begin{proof}
Let $f \colon A^n \to A$ be a member of $\mathcal{C}$.
We split the analysis in different cases according to the number of essential arguments of $f$.

Consider first the case that $f$ depends on all of its arguments or on none of them. Then $f$ is totally symmetric, and hence $f$ has a unique identification minor. By Theorem~\ref{thm:totsymm}, $f$ is reconstructible, and it follows from Lemma~\ref{lem:recsetrec} that $f$ is set-reconstructible.

Consider then the case that $f$ has $r$ essential arguments with $2 \leq r \leq n - 1$.
Then $f$ equals, up to permutation of arguments, $\circ^{(n)}_r$, and the deck of $f$ comprises
$\binom{r}{2}$ copies of $\circ_{r-1}$ and $\binom{n}{2} - \binom{r}{2}$ copies of $\circ_r$.
Thus, the set-deck of $f$ is $\{\circ_r, \circ_s\}$, where $s = r - 1$ if $(A;{\circ})$ is a semilattice and $s = r - 2$ if $(A; {\circ})$ is a Boolean group and $\circ_0$ stands for the constant function $0$.
Let $g \colon A^n \to A$ be a set-reconstruction of $f$.

Suppose, on the contrary, that $g$ depends on all of its arguments. If $n = 4$ and $k = 2$, then we have a contradiction with Lemma~\ref{lem:4-aryBoolean}. If $n \geq \max(k, 3) + 2$, then it follows from Theorem~\ref{thm:Willard2.6} that $g$ is totally symmetric. Hence $g$ has a unique identification minor. This contradicts the fact that the set-deck of $f$ is not a singleton.

We can thus assume that $g$ has inessential arguments. In this case, $g$ is a card of $g$. This implies that $g$ is equivalent to $\circ_s$ or $\circ_r$.
The former case is impossible, because $\circ_r$ is not a minor of $\circ_s$ (consider the number of essential arguments).
Consequently, $f \equiv g$, i.e., $f$ is set-reconstructible.

Finally, consider the case that $f$ has exactly one essential argument. Corollary~\ref{thm:GSL} shows that $f$ is reconstructible, which implies that $f$ is set-reconstructible by Lemma~\ref{lem:recsetrec}.
\end{proof}

\begin{theorem}
\label{thm:setrecPost}
Let $\mathcal{C}$ be a clone on $\{0, 1\}$.
If $\mathcal{C}$ is included in $\Lambda$, $V$ or $L$, then $\mathcal{C}^{(\geq 4)}$ is set-reconstructible.
Otherwise, $\mathcal{C}^{(\geq n)}$ is not weakly reconstructible for any $n \geq 1$, and $\mathcal{C}$ contains pairs of nonequivalent strongly hypomorphic functions of arbitrarily high arity.
\end{theorem}

\begin{proof}
Follows immediately from Theorem~\ref{thm:recPost} and Proposition~\ref{prop:setrecclones}.
\end{proof}

%%%%%%%%%%%%%%%%%%%%%%%%%%%%%%%%%%%%%%%%%%%%%%%%%%

\end{document}